\documentclass[11pt,reqno,a4paper]{amsart}
\usepackage[normalem]{ulem}
\usepackage[colorlinks]{hyperref}
\usepackage{amsaddr}
\usepackage{amssymb,epsfig,graphics}
\usepackage{amsmath}
\usepackage{enumerate}
\usepackage{tikz}
\usetikzlibrary{arrows}

\newcommand{\vertiii}[1]{{\left\vert\kern-0.25ex\left\vert\kern-0.25ex\left\vert #1 
    \right\vert\kern-0.25ex\right\vert\kern-0.25ex\right\vert}}

\newtheorem{theorem}{Theorem}[section]
\newtheorem{proposition}[theorem]{Proposition}

\newtheorem{lemma}[theorem]{Lemma}

\newtheorem{remark}[theorem]{Remark}

\theoremstyle{definition}

\DeclareMathOperator{\logg}{logg}

\def\1{{{\mathit 1} \!\!\>\!\! I} }

\title[Diffusion coefficient of intermittent maps]{Differentiability of the diffusion coefficient for a family of intermittent maps}
\author{Fanni M. Sélley}
\address{Leiden University Mathematical Institute, Niels Bohrweg 1 2333 CA Leiden}
\email{f.m.selley@math.leidenuniv.nl}
\thanks{\emph{ORCID of the author}: \url{https://orcid.org/my-orcid?orcid=0000-0002-8694-6715}}

\begin{document}

\maketitle

\begin{abstract}
It is well known that the Liverani--Saussol--Vaienti map satisfies a central limit theorem for Hölder observables in the parameter regime where the correlations are summable. We show that when $C^2$ observables are considered, the variance of the limiting normal distribution is a $C^1$ function of the parameter. We first show this for the first return map to the base of the second branch by studying the Green-Kubo formula, then conclude the result for the original map using Kac's lemma and relying on linear response.
\end{abstract}

\section{Introduction}

We consider a one-parameter family $\{f_{\alpha}\}_{\alpha \in \mathcal{A}}$ of Pomeau--Manneville type maps introduced by Liverani, Saussol and Vaienti \cite{LSV99}. Provided that $\alpha \in [0,1)$, each map preserves a unique probability measure $\nu_{\alpha}$ and exhibits polynomial decay of correlations for H\"older observables \cite{LSV99,Y99}. Restricting $\alpha$ further to $[0,1/2)$, the central limit theorem also holds: ergodic sums of centered observables, normalized by $\sqrt{n}$, converge in distribution to a Gaussian random variable with mean zero and variance $\sigma^2(\alpha)$ \cite{Y99,G15}. In this paper we study the smoothness of the mapping $\alpha \mapsto \sigma^2(\alpha)$. 

The question is motivated by the deterministic diffusion on the real line defined by \cite{KK07}. Diffusion can be characterized by the linear growth rate of the mean square displacement of an ensemble of moving particles, called the \emph{diffusion coefficient}. A short calculation shows that the said system on the real line is conjugated to a Pomeau--Manneville type map of the interval, and the corresponding quantity of the diffusion coefficient is $\sigma^2(\alpha)$ associated to a particular observable. So knowing the regularity of $\sigma^2(\alpha)$ gives information about the parameter dependence of the diffusion itself. 

Continuity properties of the diffusion coefficient associated to different types of dynamics were studied extensively, for an overview see the book \cite{K07} and references therein. In \cite{FG82} a diffusion on the real line was defined where trajectories for most time are localized to boxes $B_N = [N,N+1) \times [N,N+1)$, but at certain instances they can move to either $B_{N-1}$ or $B_{N+1}$. The analytic expression of the map is such that the system on the real line is conjugated to an interval map, which aides calculations a great deal. Focusing on this simple setting, the piecewise linear case was studied in \cite{KD95,KK03,KKH08} and low regularity was reported both from the analytic and geometric measure theory perspective. In \cite{KKH08} the setting was generalized to a family of maps exhibiting exponential decay of correlations with uniform parameters, and log-Lipschitz continuity of the diffusion coefficient was proved with methods relying mainly on the transfer operator approach. The paper \cite{KK07}, providing our motivation, proposed to consider a family of intermittent maps and study the so-called generalized diffusion coefficient in the parameter regime of anomalous diffusion. In the regime of normal diffusion \cite{BCV16} proved the continuity of the diffusion coefficient for a wide variety of ($f_{\alpha}$-independent) potentials. In fact, for the zero potential (measure of maximal entropy) they proved that the diffusion coefficient as a function of $f_{\alpha}$ varies in a continuously differentiable way.

In this paper we consider the regime of normal diffusion, i.e. $\alpha \in [0,1/2)$ (and the SRB measure $\nu_{\alpha}$). Then the variance $\sigma^2(\alpha)$ can be given as the sum of correlations by the Green--Kubo formula, so the smoothness of correlations gives a good guess for the smoothness of $\sigma^2(\alpha)$. The first (self-correlation) term in the formula is $\int \psi^2 d\nu_{\alpha}$. The smooth differentiability of the mapping $\alpha \mapsto \int \psi^2 d\nu_{\alpha}$ is the question of linear response, which is well understood for the family of maps in question \cite{BT16,K16}. The higher order terms $\int \psi \cdot \psi \circ T_{\alpha}^k \:d\nu_{\alpha}$ have an additional $\alpha$-dependence in the integrand $\psi \cdot \psi \circ T_{\alpha}^k$, however since $\alpha \mapsto T_{\alpha}^k$ is smooth, the smoothness of $\alpha \mapsto \int \psi \cdot \psi \circ T_{\alpha}^k \: d\nu_{\alpha}$ follows. Continuity of $\sigma^2(\alpha)$ follows from the summability of correlations, but in order to show continuous differentiability, one also has to show that the $\alpha$-derivatives of the correlations are summable.  Making these steps precise we show that smooth differentiability also holds for the higher order terms of the Green--Kubo formula, furthermore,  $\alpha \mapsto \sigma^2(\alpha) \in C^1[\alpha_-,\alpha_+]$ for any $0< \alpha_- < \alpha_+< 1/2$. 

The idea of the proof is to define the usual first return map, prove smooth differentiability for the corresponding variance $\tilde{\sigma}^2(\alpha)$ and conclude the same regularity for $\sigma^2(\alpha)$ by using Kac's formula. The calculations make use of the linear response result and technical estimates of \cite{K16}, but the issues sketched above require work that goes beyond being a mere corollary of linear response. 

We note that another important quantity characterizing diffusion is the \emph{drift coefficient} (the expectation of the normal random variable given by the central limit theorem in case of a non-centered observable $\varphi$). This translates to the integral $\int \varphi \: d\nu_{\alpha}$, hence the smooth differentiability as a function of $\alpha$ is covered by the linear response result of \cite{BT16,K16}.

The structure of this paper is as follows: in Section 2 we introduce our setting, main results and give a sketch of the proof. Section 3 contains the proof of our main theorem. Section 4 is devoted to concluding remarks on possible further directions of research. 
%Appendix A provides some technical estimates for the sake of completeness.
\vspace{0.3cm}\\
\textbf{Acknowledgements.} I would like to express my gratitude to Wael Bahsoun for providing me perspective and guidance for this work. I would also like to thank Alexey Korepanov, Julien Sedro and Dalia Terhesiu helpful discussions. 
\vspace{0.3cm}\\
\textbf{Notational remark.} Throughout these notes, we will denote the Lebesgue measure on $[0,1]$ by $m$. Furthermore $C> 0$ will denote a generic constant whose value might change from one line to the next.

\section{Setting and main result}

We consider the Liverani--Saussol--Vaienti map $f_{\alpha}: [0,1] \to [0,1]$ defined as
\begin{equation}
f_{\alpha}(x)=\begin{cases}
x(1+2^{\alpha}x^{\alpha}) &\quad \text{for} \quad 0 \leq x \leq1/2 \\
2x-1 &\quad \text{for} \quad 1/2 < x \leq 1.
\end{cases}
\end{equation}
It is well known that this map preserves a unique (Lebesgue) absolutely continuous measure $\nu_{\alpha}$, with density function $\rho_{\alpha}$ \cite{LSV99}.

The following theorem is also well known, see for example \cite[Theorem 6]{Y99} or \cite[Theorem 4.1]{G15}.
\begin{theorem} \label{thm:clt}
Let $0 \leq \alpha < 1/2$ and $\psi: [0,1] \to \mathbb{R}$ a $C^2$ function. Denote $S_n \psi=\sum_{i=0}^{n-1}\psi\circ f_{\alpha}^i$ and $\hat \psi_{\alpha}=\psi - \int \psi \rho_{\alpha} dm$. Then there exists $\sigma \geq 0$ such that $\frac{S_n \hat \psi_{\alpha}}{\sqrt{n}}$ converges in distribution to a random variable $\mathcal{N}(0,\sigma^2)$.
\end{theorem}

Actually, \cite[Theorem 6]{Y99} considers $\psi$ that is Hölder continuous, while \cite[Theorem 4.1]{G15} is stated for $C^1$ observables. We state this weaker version of both theorems for $C^2$ observables, as this is the setting that we will use.

Obviously $\sigma^2=\sigma^2(\alpha)$. The goal of these notes is to study the smoothness of the mapping $\alpha \mapsto \sigma^2(\alpha)$. Our main statement is the following:
\begin{theorem} \label{thm:main}
Let $0 < \alpha_- < \alpha_+ < 1/2$. In the setting of Theorem \ref{thm:clt},
\[
\alpha \mapsto \sigma^{2}(\alpha) \in C^1([\alpha_{-},\alpha_{+}]).
\]
\end{theorem}
In the rest of this section we give the outline of the proof. 

Let $\tau_{\alpha}(x)=\min\{n \geq 1: T_{\alpha}^k(x) \in (1/2,1]\}$ be the return time to the interval $(1/2,1]$. Define the induced map $F_{\alpha}: (1/2,1] \to (1/2,1]$ as
\begin{equation}
F_{\alpha}(x)=f_{\alpha}^{\tau_{\alpha}(x)}(x),
\end{equation}
and the induced observable
\[
\hat \Psi_{\alpha}(x)=\sum_{k=0}^{\tau_{\alpha}(x)-1}\hat \psi_{\alpha}(f_{\alpha}^k(x)).
\]
As a slight abuse of notation, we denote the ergodic sums of $\hat \Psi_{\alpha}$ under $F_{\alpha}$ also by $S_n \hat  \Psi_{\alpha}$, that is,
\[
S_n \hat  \Psi_{\alpha} = \sum^{n-1}_{i=0} \hat \Psi_{\alpha} \circ F^i_{\alpha}.
\]
Since $F_{\alpha}$ is a mixing Gibbs--Markov map (see \cite[Lemma 3.60]{A20}) with a unique invariant measure $\mu_{\alpha}$ (see \cite[Lemma 4.4.1]{A97}), and $\hat \Psi_{\alpha} \in L^2([1/2,1])$, we can conclude that the induced map also satisfies a Central Limit Theorem in the sense that $\frac{S_n \hat \Psi_{\alpha}}{\sqrt{n}}$ converges in distribution to a random variable $\mathcal{N}(0, \tilde{\sigma}^2(\alpha))$. 

Using \cite[Theorem 4.8]{G15} and Kac's lemma, we show that
\begin{equation} \label{eq:varcon}
\sigma^2(\alpha)=\frac{\tilde{\sigma}^2(\alpha)}{\int \tau_{\alpha} d \mu_{\alpha}},
\end{equation}

By \cite{K16} (see in particular the proof of \cite[Theorem 1.1]{K16}), the map $\alpha \mapsto \int \tau_{\alpha} d\mu_{\alpha}$ is continuously differentiable (and trivially bounded from below), hence it is sufficient to study the smoothness of  $\alpha \mapsto \tilde{\sigma}^2(\alpha)$. 

\begin{remark}
By Kac's lemma we also obtain that
\[
\int \hat \Psi_{\alpha} d\mu_{\alpha}=\frac{\int \hat \psi_{\alpha} d\nu_{\alpha}}{\int \tau_{\alpha}d\mu_{\alpha}}=0.
\]
\end{remark}

We express $\tilde{\sigma}^2(\alpha)$ with the Green--Kubo formula:
\begin{equation} \label{eq:gk}
\tilde{\sigma}^2(\alpha)=\int \hat \Psi_{\alpha}^2 \:d\mu_{\alpha} + 2 \sum_{k \geq 1} \int \hat \Psi_{\alpha} \cdot \hat \Psi_{\alpha} \circ F_{\alpha}^k\: d\mu_{\alpha}.
\end{equation}
We first prove (in Proposition \ref{prop:cont}) that all correlations are individually continuously differentiable functions of $\alpha$ in the sense that
\begin{equation} \label{eq:gkterms}
\begin{split}
&\alpha \mapsto \int \hat \Psi_{\alpha}^2 \:d\mu_{\alpha} \in C^1([\alpha_-,\alpha_+]) \\
&\alpha \mapsto \int \hat \Psi_{\alpha} \cdot \hat \Psi_{\alpha} \circ F_{\alpha}^k\: d\mu_{\alpha} \in C^1([\alpha_-,\alpha_+]), \: k \geq 1.
\end{split}
\end{equation} 
We then show (in Proposition \ref{prop:sum}) that the series 
\begin{equation} \label{eq:gksums}
\sum_{k \geq 1} \int \hat \Psi_{\alpha} \cdot \hat \Psi_{\alpha} \circ F_{\alpha}^k\: d\mu_{\alpha} \quad \text{and} \quad \sum_{k \geq 1} \partial_{\alpha}\left(\int \hat \Psi_{\alpha} \cdot \hat \Psi_{\alpha} \circ F_{\alpha}^k\: d\mu_{\alpha}\right)
\end{equation}
converge uniformly, from which we can conclude that $$\alpha \mapsto \tilde{\sigma}^{2}(\alpha) \in C^1([\alpha_{-},\alpha_{+}]),$$ finishing the proof of Theorem \ref{thm:main}.

In the next section we will make these steps precise.

\section{Proof}
In the first part of this section we argue the central limit theorem for the induced map and Equation \eqref{eq:varcon} giving the connection between the variances $\sigma^2$ and $\tilde{\sigma}^2$. These results are part of the general folklore, but we give an argument for the sake of completeness as exact references are hard to track down in the literature. 

In the second part we study the Green--Kubo formula \eqref{eq:gk} in depth, proving \eqref{eq:gkterms} and \eqref{eq:gksums}. The calculation makes use of the linear response result of \cite{K16}, and takes things a few steps further to obtain continuous differentiability of $\tilde{\sigma}(\alpha)$.  

\subsection{Central limit theorem for the induced map.}

Let $x_0=1$, set $x_{n+1}=f_{\alpha}^{-1}(x_n) \cap [0,1/2]$ and $y_{n+1}=f_{\alpha}^{-1}(x_n) \cap (1/2,1]$. Let  $I_n=(y_{n+1},y_n)$ and $J_n=(x_{n+1},x_n)$. Define the partition $\pi_I$ as the collection of intervals $\{I_n\}_{n \in \mathbb{N}}$, and similarly $\pi_J$ as $\{J_n\}_{n \in \mathbb{N}}$. Note that for $z \in I_n$, $\tau_{\alpha}(z)=n$ -- meaning that $\pi_I$ partitions $(1/2,1]$ according to their return time.

Let $s(x,y)$ be the separation time of $x, y \in (1/2,1]$ under $F_{\alpha}$: 
\[
s(x,y)=\inf\{n \in \mathbb{N}: \text{$F_{\alpha}^n(x)$ and $F_{\alpha}^n(y)$ are in separate elements of $\pi_I$}\}
\]
We extend this definition to $[0,1/2]$ as follows: let $x,y \in [0,1/2]$ and denote by $x',y'$ their first returns to $(1/2,1]$. If $f_{\alpha}^i(x)$ and $f_{\alpha}^i(y)$ stay in the same partition elements of $\pi_J$ until their first return to $(1/2,1]$, set $s(x,y)=s(x',y')+1$, otherwise $s(x,y)=0$. 

Define $d_{\theta}(x,y)=\theta^{-s(x,y)}$ for some $\theta \in (1,2]$. Then $((1/2,1], \mu_{\alpha}, F_{\alpha})$ is also Gibbs--Markov for the metric $d_{\theta}$. \footnote{Indeed, expansion is evident since $\theta^{-s(F_{\alpha}(x),F_{\alpha}(y))}=\theta^{-s(x,y)+1}=\theta \cdot \theta^{-s(x,y)}$ and we get the distortion estimate by using $|F_{\alpha}(x)-F_{\alpha}(y)| \leq C \theta^{-s(F_{\alpha}(x),F_{\alpha}(y))}$.}

%Finally, since $F_{\alpha}$ is locally eventually onto, it is topologically mixing. This implies that the invariant density is %bounded from above and below, and the system is exact -- in particular mixing (see the introduction of \cite[Section 4.7]{A97})

We now study the induced observable. Let $g$ be a function from $I_n \in \pi_I$ to $\mathbb{R}$. Define the Lipschitz semi-norm of $g$ as
\[
|g|_{Lip(I_n)}=\inf\{ C> 0:   \forall x,y \in I_n \quad |g(x)-g(y)| \leq C\theta^{-s(x,y)}\}.
\]
\begin{proposition} \label{prop:l2}
$\hat \Psi_{\alpha} \in L^2$ and $\sum_{r=1}^{\infty} \mu_{\alpha}(I_r)|\hat \Psi_{\alpha}|_{Lip(I_r)} < \infty$.
\end{proposition}
\begin{proof}
We first show the induced observable is in $L^2$. Recall that
\[
\hat \Psi_{\alpha}(x)=\sum_{k=0}^{\tau_{\alpha}(x)-1}\hat \psi_{\alpha}(f_{\alpha}^k(x)).
\]
and $\hat \psi_{\alpha} \in C^2$. Write
\begin{align*}
\hat \Psi_{\alpha}(x)&=\hat \Psi_{\alpha}(x)-\tau_{\alpha}(x)\hat \psi_{\alpha}(0)+\tau_{\alpha}(x)\hat \psi_{\alpha}(0)\\
&=\tau_{\alpha}(x)\hat \psi_{\alpha}(0)+\sum_{k=0}^{\tau_{\alpha}(x)-1}\hat \psi_{\alpha}(f_{\alpha}^k(x))-\hat \psi_{\alpha}(0),
\end{align*}
We have $x_{n+1}-x_n \sim 1/n^{1+1/\alpha}$ (for instance by \cite[Equation (3. 149)]{A20}), thus 
\begin{align*}
\int \tau_{\alpha}^2 \asymp \sum_{k \geq 1} k^2 \cdot (x_{k+1}-x_k) \asymp \sum_{k \geq 1} k^2 \cdot 1/k^{1+1/\alpha} < \infty
\end{align*}
for all $\alpha \in (0,1/2)$. For the other term, we first exploit that $\psi$ is Lipschitz continuous with constant $C$. This implies that $\hat \psi_{\alpha}$ is Lipschitz with constant $C$, hence
\begin{align*}
\sum_{k=0}^{\tau_{\alpha}(x)-1}|\hat \psi_{\alpha}(f_{\alpha}^k(x))-\hat \psi_{\alpha}(0)| \leq C \sum_{k=0}^{\tau_{\alpha}(x)-1}|f_{\alpha}^k(x)|.
\end{align*} 
For $x \in (y_{n+1},y_n)$, we have $f_{\alpha}^i(x) \in (x_{n+1-i},x_{n-i})$, thus $|f_{\alpha}^i(x)| \leq c (n-i)^{-1/\alpha}$. Then
\[
|\hat \Psi_{\alpha}(x) - \tau_{\alpha}(x) \cdot \hat \psi_{\alpha}(0)| \leq c+c\sum_{j=1}^{n-1}(1/j)^{1/\alpha} \leq c n^{1-1/\alpha}
\]
and 
\[
\int (\hat \Psi_{\alpha} - \tau_{\alpha} \cdot \hat \psi_{\alpha}(0))^2 \asymp \sum_{k \geq 1} k^{2-2/\alpha} \cdot 1/k^{1+1/\alpha} = \sum_{k \geq 1} k^{1-3/\alpha},
\]
which is summable for all $\alpha \in (0,1)$. Thus $\hat \Psi_{\alpha} \in L^2$.

We now prove the second statement. Since the induced map on $(1/2,1]$ is uniformly expanding with expansion factor $\lambda = 2$, it is clear that $|x-y| < C\theta^{-s(x,y)}$. Using that $\hat \psi_{\alpha}$ is Lipschitz continuous, we get
$|\hat \psi_{\alpha}(x)-\hat \psi_{\alpha}(y)| \leq C |x-y| \leq C \theta^{-s(x,y)}$.

By the definition of the separation time, we can see that for $x,y \in I_r$,
\[
\sum_{k=0}^{r-1}|\hat \psi_{\alpha}(f_{\alpha}^k(x))-\hat \psi_{\alpha}(f_{\alpha}^k(y))| \leq \sum_{k=0}^{r-1}C\theta^{-(f_{\alpha}^k(x),f_{\alpha}^k(y))} \leq Cr\theta^{-s(x,y)}
\]
and thus $|\hat \Psi_{\alpha}|_{Lip(I_r)} \leq Cr$. Furthermore, since $d\mu_{\alpha}=h_{\alpha}dm$, where $h_{\alpha} \in L^{\infty}$, we have $|h_{\alpha}| \leq C_{\alpha}$ and
\[
|\mu_{\alpha}(I_r)| \leq C|m(I_r)| \leq C(1/r)^{1+1/\alpha},
\]
similarly to previous computations. This gives
\[
\sum_{r \geq 1} \mu_{\alpha}(I_r)|\hat \Psi_{\alpha}|_{Lip(I_r)} \leq C\sum_{r \geq 1} r^{-1/\alpha} < \infty.
\]
\end{proof}
The central limit theorem for the induced map (as stated below) follows by \cite[Theorem 1.4]{D14}.
\begin{theorem} \label{thm:induced}
\[
\frac{S_n \hat \Psi_{\alpha}}{\sqrt{n}} \to \mathcal{N}(0, \tilde{\sigma}^2) \quad \text{in distribution},
\]
and
\[
\tilde{\sigma}^2=\int \hat \Psi_{\alpha}^2 \:d\mu_{\alpha} + 2 \sum_{k \geq 1} \int \hat \Psi_{\alpha} \cdot \hat \Psi_{\alpha} \circ F_{\alpha}^k\: d\mu_{\alpha}.
\]
\end{theorem} 

Next we prove the relation between the variances corresponding to the original and the induced map.

\begin{proposition}
\[
\sigma^2(\alpha)=\frac{\tilde{\sigma}^2(\alpha)}{\int \tau_{\alpha} d \mu_{\alpha}}
\]
\end{proposition}

\begin{proof}
Let $Y= (1/2,1]$. Recall that
\[
\frac{S_n \hat \Psi_{\alpha}}{\sqrt{n}} \to \mathcal{N}(0, \tilde{\sigma}^2(\alpha)).
\]
We first show that
\begin{equation}
\frac{S_n \tau_{\alpha}-n/\nu_{\alpha}(Y)}{\sqrt{n}} \text{ is tight and } \sup_{0 \leq k \leq \sqrt{n}}\frac{|S_k \hat \Psi_{\alpha}|}{\sqrt{n}} \text{ tends to 0 in probability.}
\end{equation}
Indeed, consider the observable $g$ equal to $1-1/\nu_{\alpha}(Y)$ on $Y$ and 1 elsewhere (this observable is not $C^2$, but $C^2$ on [0,1/2] and [1/2,1] which is sufficient). Then the corresponding induced observable is $\tau_{\alpha}-1/\nu_{\alpha}(Y)$, and by Theorem \ref{thm:induced}, $\frac{S_n \tau_{\alpha}-n/\nu_{\alpha}(Y)}{\sqrt{n}}$ converges to a Gaussian random variable, hence the sequence $\frac{S_n \tau_{\alpha}-n/\nu_{\alpha}(Y)}{\sqrt{n}}$ is tight. To show that $\sup_{0 \leq k \leq \sqrt{n}}\frac{|S_k \hat \Psi_{\alpha}|}{\sqrt{n}}$ tends to 0 in probability, it suffices to show that $\frac{S_k \hat \Psi_{\alpha}}{k}$ tends almost surely to 0. This follows from the Birkhoff ergodic theorem, since $\hat \Psi_{\alpha}$ is integrable. 

Thus by \cite[Theorem 4.8]{G15}, we have
\[
\frac{S_n \hat \psi_{\alpha}}{\sqrt{\lfloor n \nu_{\alpha}(Y)\rfloor}} \to \mathcal{N}(0, \tilde{\sigma}^2(\alpha))
\]
By recalling that $\frac{S_n \hat \psi_{\alpha}}{\sqrt{n}} \to \mathcal{N}(0, \sigma^2(\alpha))$, we obtain that $\tilde{\sigma}^2(\alpha)\nu_{\alpha}(Y)=\sigma^2(\alpha)$. By Kac's formula, $\nu_{\alpha}(Y) = \frac{1}{\int \tau_{\alpha}d\mu_{\alpha}}$ giving the result. 

\end{proof}
\subsection{Analysis of the Green--Kubo formula.}

In this section we will prove that 
\[
\alpha \mapsto \int \hat \Psi_{\alpha}^2 h_{\alpha}\:dm + 2 \sum_{k \geq 1} \int \hat \Psi_{\alpha} \cdot \hat \Psi_{\alpha} \circ F_{\alpha}^k \cdot h_{\alpha}\:dm \in C^1([\alpha_{-},\alpha_+])
\]
for any $0< \alpha_{-} < \alpha_+ < 1/2$. As a first step we prove that the correlations are continuously differentiable.

We first recall a key lemma from \cite{K16} which we will use on many occasions. To state this lemma we introduce some notation. Denote by $g_{\alpha}$ the inverse of the left branch of $f_{\alpha}$, and for $z \in [0,1]$ let $z_r = g_{\alpha}^r(z)$. Furthermore, denote $F_{\alpha}|_{I_n} = F_{\alpha,n}$ and $G_{\alpha,n}=(F_{\alpha,n}^{-1})'$. Finally, define
\[
\logg x = \begin{cases}
1 & \quad x \leq e \\
\log x &\quad x>e
\end{cases}
\]
where $\log$ is the logarithm with base $e$.
\begin{lemma} \label{lem:k16} For $n \in \mathbb{N}$,
\begin{enumerate}
\item[$(K0)$] $|z'_n| \leq 1$ {\cite[Equation 5.5]{K16}}
\item[$(K1)$] $|z'_n| \leq Cn^{-(\alpha+1)/\alpha}$ {\cite[Lemma 5.3]{K16}}
\item[$(K2)$] $|z''_n/z'_n| \leq C \Rightarrow \|G'_{\alpha,n}/G_{\alpha,n}\|_{\infty} \leq C$ {\cite[Lemma 5.4]{K16}}
\item[$(K3)$] $|z'''_n/z'_n| \leq C \Rightarrow \|G''_{\alpha,n}/G_{\alpha,n}\|_{\infty} \leq C$ {\cite[Lemma 5.5]{K16}}
\item[$(K4)$] $|\partial_{\alpha}z_n| \leq Cn^{-1/\alpha}(\logg n)^2$ {\cite[Lemma 5.6]{K16}}
\item[$(K5)$] $|\partial_{\alpha}z'_n/z'_n| \leq C(\logg n)^3 \Rightarrow \|\partial_\alpha G_{\alpha,n}/G_{\alpha,n}\|_{\infty} \leq C(\logg n)^3$ {\cite[Lemma 5.7]{K16}}
\item[$(K6)$] $\|\partial_\alpha F^{-1}_{\alpha,n}\|_{\infty} \leq Cn^{-1/\alpha}(\logg n)^2 \leq C(\logg n)^2$ {\cite[Lemma 5.6 and Equation 5.1]{K16}}
\end{enumerate}
\end{lemma}

Recall that the transfer operator $P$ of a nonsingular map $T: I \to I$ is defined as the left adjoint of the composition (Koopman operator), that is,
\[
\int_I \varphi \cdot P \psi \text{ d}m =  \int_I \varphi \circ T \cdot \psi \text{ d}m, \qquad \varphi \in L^{\infty}, \psi \in L^1. 
\]
We will denote the transfer operator of $T_{\alpha}$ by $P_{\alpha}$.

\begin{proposition} \label{prop:cont}
\[
\alpha \mapsto \int \hat \Psi_{\alpha} \cdot \hat \Psi_{\alpha} \circ F_{\alpha}^k \cdot h_{\alpha}  \:dm \in C^1([\alpha_-,\alpha_+])
\]
for all $k \geq 0$.
\end{proposition}

\begin{proof}
Write
\begin{align*}
\int \hat \Psi_{\alpha} \cdot \hat \Psi_{\alpha} \circ F_{\alpha}^k \cdot h_{\alpha}  \:dm &= \int P_{\alpha}^k(\hat \Psi_{\alpha}h_{\alpha}) \cdot \hat \Psi_{\alpha}  \:dm \\
&= \int P_{\alpha}\left(P_{\alpha}^k(\hat \Psi_{\alpha}h_{\alpha}) \cdot \hat \Psi_{\alpha} \right)  \:dm
\end{align*}
For $\underline{n}=(n_1,\dots,n_k)$ let $F_{\alpha, \underline{n}}=F_{\alpha,n_1} \circ \dots \circ F_{\alpha,n_k}$ and $G_{\alpha, \underline{n}}=(F_{\alpha, \underline{n}}^{-1})'$. With this notation we can write
\begin{align}
&\int P_{\alpha}\left(P_{\alpha}^k(\hat \Psi_{\alpha}h_{\alpha}) \cdot \hat \Psi_{\alpha} \right)  \:dm \nonumber \\
&= \int \sum_{n_k=1}^{\infty}\dots \sum_{n_0=1}^{\infty} (h_{\alpha} \circ F_{\alpha, \underline{n}}^{-1} \circ F_{\alpha, n_0}^{-1}) (\hat \Psi_{\alpha} \circ F_{\alpha, \underline{n}}^{-1} \circ F_{\alpha, n_0}^{-1})  (\hat \Psi_{\alpha} \circ F_{\alpha, n_0}^{-1})  \nonumber \\
&\qquad \times G_{\alpha, \underline{n}} \circ F_{\alpha, n_0}^{-1} \cdot G_{\alpha,n_0}  \:dm \label{eq:corrdet}
\end{align}
and
\begin{align}
&\partial_{\alpha} \int P_{\alpha}\left(P_{\alpha}^k(\hat \Psi_{\alpha}h_{\alpha}) \cdot \hat \Psi_{\alpha} \right)  \:dm \nonumber \\
&= \int \sum_{n_k=1}^{\infty}\dots \sum_{n_0=1}^{\infty} \partial_{\alpha}\{(h_{\alpha} \circ F_{\alpha, \underline{n}}^{-1} \circ F_{\alpha, n_0}^{-1}) (\hat \Psi_{\alpha} \circ F_{\alpha, \underline{n}}^{-1} \circ F_{\alpha, n_0}^{-1})  (\hat \Psi_{\alpha} \circ F_{\alpha, n_0}^{-1}) \nonumber \\
&\qquad \times G_{\alpha, \underline{n}} \circ F_{\alpha, n_0}^{-1} \cdot G_{\alpha,n_0}\}  \:dm \label{eq:corrdet2}
\end{align}
It is clear that the summands in \eqref{eq:corrdet} are jointly continuous functions of $\alpha$ and the spatial variable $x$, and this also holds for their partial derivatives with respect to $\alpha$ appearing in \eqref{eq:corrdet2}. Indeed, continuity of terms involving $h_{\alpha}$ follows by \cite[Theorem 2.1]{K16}), and for the terms involving $\hat \Psi_{\alpha}$ we rely on \cite[Theorem 1.1]{K16}).

In the rest of the proof we show that the series
\begin{align}
&\sum_{n_k=1}^{\infty}\dots \sum_{n_0=1}^{\infty} (h_{\alpha} \circ F_{\alpha, \underline{n}}^{-1} \circ F_{\alpha, n_0}^{-1}) (\hat \Psi_{\alpha} \circ F_{\alpha, \underline{n}}^{-1} \circ F_{\alpha, n_0}^{-1})  (\hat \Psi_{\alpha} \circ F_{\alpha, n_0}^{-1}) \nonumber \\
&\hspace{2.1cm} \times G_{\alpha, \underline{n}} \circ F_{\alpha, n_0}^{-1} \cdot G_{\alpha,n_0}  \label{eq:kfold} \\
&\sum_{n_k=1}^{\infty}\dots \sum_{n_0=1}^{\infty} \partial_{\alpha}\{(h_{\alpha} \circ F_{\alpha, \underline{n}}^{-1} \circ F_{\alpha, n_0}^{-1}) (\hat \Psi_{\alpha} \circ F_{\alpha, \underline{n}}^{-1} \circ F_{\alpha, n_0}^{-1})  (\hat \Psi_{\alpha} \circ F_{\alpha, n_0}^{-1}) \nonumber \\
&\hspace{2.1cm} \times G_{\alpha, \underline{n}} \circ F_{\alpha, n_0}^{-1} \cdot G_{\alpha,n_0}\} \label{eq:kfolda}
\end{align} 
converge uniformly. 

We first show that \eqref{eq:kfold} converges uniformly. This will imply the continuity of $\alpha \mapsto \int \hat \Psi_{\alpha} \cdot \hat \Psi_{\alpha} \circ F_{\alpha}^k \cdot h_{\alpha}  \:dm$. According to \cite[Theorem 2.1]{K16}, $h_{\alpha} \in C^2([1/2,1])$, hence $\|h_{\alpha}\|_{\infty} \leq K$ where $K$ does not depend on $\alpha$.	

First compute that
\begin{align*}
(\hat \Psi_{\alpha} \circ F_{\alpha, n_0}^{-1})(z)&=\hat \psi_{\alpha}\left(\frac{z_{n_0}+1}{2} \right)+\sum_{j=0}^{n_0-1} \hat \psi_{\alpha}(f_{\alpha}^j(z_{n_0}))\\
&=\hat \psi_{\alpha}\left(\frac{z_{n_0}+1}{2} \right)+\sum_{j=1}^{n_0} \hat \psi_{\alpha}(z_{j})
\end{align*}
giving $\|\hat \Psi_{\alpha} \circ F_{\alpha, n_0}^{-1}\|_{\infty} \leq C \|\hat \psi_{\alpha}\|_{\infty} (n_0+1) \leq 2C \| \psi\|_{\infty} (n_0+1)$. The same argument gives also $\|\hat \Psi_{\alpha} \circ F_{\alpha, \underline{n}}^{-1} \circ F_{\alpha, n_0}^{-1}\|_{\infty} \leq C (n_0+1)$.

We have $\|G_{\alpha,n_0}\|_{\infty}=\frac{1}{2} \sup_z |z_{n_0}'| \leq C n_0^{-\left(1+\frac{1}{\alpha}\right)}$ by (K1) from Lemma \ref{lem:k16}. Furthermore,
\[
G_{\alpha, \underline{n}}=(F_{\alpha,n_k}^{-1} \circ \dots \circ F_{\alpha,n_1}^{-1})'= \Pi_{i=1}^k (F_{\alpha,n_i}^{-1})' (F_{\alpha,n_{i-1}}^{-1} \circ \dots \circ F_{\alpha,n_1}^{-1})
\]
hence $\|G_{\alpha,\underline{n}}\|_{\infty} \leq \|G_{\alpha,n_1}\|_{\infty}\dots \|G_{\alpha,n_k}\|_{\infty} \leq C (n_1\dots n_k)^{-\left(1+\frac{1}{\alpha}\right)}$ (note that $C=C(k)$, but this causes no issue for the present argument as we view $k$ fixed.)

Putting all this together, we can bound the $k$-fold sum in \eqref{eq:corrdet} by
\[
C\sum_{n_k=1}^{\infty}\dots \sum_{n_0=1}^{\infty} (n_0+1)^2 (n_0n_1\dots n_k)^{-\left(1+\frac{1}{\alpha}\right)},
\]
which is finite if $\alpha \in (0,1/2)$.

To show that \eqref{eq:kfolda} converges uniformly, we first study the $\partial_\alpha$ partial derivatives of the summands in \eqref{eq:corrdet}, starting with
\begin{align*}
\partial_{\alpha}(h_{\alpha} \circ F_{\alpha,(\underline{n},n_0)}^{-1})=(\partial_{\alpha} h_{\alpha}) \circ F_{\alpha,(\underline{n},n_0)}^{-1} + h'_{\alpha} \circ F_{\alpha,(\underline{n},n_0)}^{-1}\partial_{\alpha} F_{\alpha,(\underline{n},n_0)}^{-1}
\end{align*}
By \cite[Theorem 2.1]{K16} $\partial_{\alpha} h_{\alpha} \in C^1([1/2,1])$ and $\|\partial_{\alpha} h_{\alpha}\|_{\infty} \leq K$, for a constant $K > 0$ independent of $\alpha$. Since
\begin{align*}
\partial_{\alpha} (F_{\alpha,n_k}^{-1} \circ \dots \circ F_{\alpha,n_0}^{-1})&=\partial_{\alpha} (F_{\alpha,n_k}^{-1}) \circ (F_{\alpha,n_{k-1}}^{-1} \circ \dots \circ F_{\alpha,n_0}^{-1})\\
&+(F_{\alpha,n_k}^{-1} \circ \dots \circ F_{\alpha,n_0}^{-1})'\partial_{\alpha}(F_{\alpha,n_{k-1}}^{-1} \circ \dots \circ F_{\alpha,n_0}^{-1}) \\
&=\partial_{\alpha} (F_{\alpha,n_k}^{-1}) \circ (F_{\alpha,n_{k-1}}^{-1} \circ \dots \circ F_{\alpha,n_0}^{-1})\\
&+G_{\alpha,(n_0,\dots, n_k)}\partial_{\alpha}(F_{\alpha,n_{k-1}}^{-1} \circ \dots \circ F_{\alpha,n_0}^{-1})
\end{align*}
we obtain by induction that
\begin{equation} \label{eq:itinv}
\|\partial_{\alpha} (F_{\alpha,n_k}^{-1} \circ \dots \circ F_{\alpha,n_0}^{-1})\|_{\infty} \leq C \sum_{j=0}^k \|\partial_{\alpha} F_{\alpha,n_j}^{-1}\|_{\infty}
\end{equation}
(where we repeatedly used the fact that $\|G_{\alpha,\underline{n}}\|_{\infty} \leq C(k+1)$ when $|\underline{n}| \leq k+1$.)

By (K6) from Lemma \ref{lem:k16} we have $\|\partial_{\alpha} F_{\alpha,n_j}^{-1}\|_{\infty} \leq C(\logg n_j)^2$ and hence
\begin{align} \label{eq:firstalpha}
\|\partial_{\alpha}(h_{\alpha} \circ F_{\alpha,(\underline{n},n_0)}^{-1})\|_{\infty} \leq C \sum_{j=0}^k (\logg n_j)^2
\end{align}
Next we compute that
\begin{align*}
\partial_{\alpha}(\hat \Psi_{\alpha} \circ F_{\alpha, n_0}^{-1})(z)=\psi'\left(\frac{z_{n_0}+1}{2} \right)\partial_{\alpha}z_{n_0}+\sum_{j=1}^{n_0} \psi'(z_{j})\partial_{\alpha}z_{n_j}-(n_0+1) \partial_{\alpha}\hat \psi_{\alpha}
\end{align*}
Since $|\partial_{\alpha}z_{n_j}| \leq C$ by (K4) from Lemma \ref{lem:k16} and $|\partial_{\alpha}\hat \psi_{\alpha}| \leq C$ by \cite[Theorem 1.1]{K16} we obtain the bound $\|\partial_{\alpha}(\hat \Psi_{\alpha} \circ F_{\alpha, n_0}^{-1})\|_{\infty} \leq C(\|\psi\|_{C^1}+1) (n_0+1)$ and by a similar argument 
\begin{equation} \label{eq:secondalpha}
\|\partial_{\alpha}(\hat \Psi_{\alpha} \circ F_{\alpha, (\underline{n},n_0)}^{-1})\|_{\infty} \leq C\|\psi\|_{C^1}(n_0+1)
\end{equation}
Next,
\[
G_{\alpha,\underline{n}} \circ F_{\alpha,n_0}^{-1}=\Pi_{i=1}^k G_{\alpha,n_i} \circ F_{\alpha,n_{i-1}}^{-1} \circ \dots \circ F_{\alpha,n_{0}}^{-1}
\]
and
\begin{align*}
\partial_{\alpha} (G_{\alpha,\underline{n}} \circ F_{\alpha,n_0}^{-1})=\sum_{j=1}^k &\partial_{\alpha}[G_{\alpha,n_j} \circ F_{\alpha,n_{j-1}}^{-1} \circ \dots \circ F_{\alpha,n_{0}}^{-1}] \\
&\times \Pi_{i \neq j} G_{\alpha,n_i} \circ F_{\alpha,n_{i-1}}^{-1} \circ \dots \circ F_{\alpha,n_{0}}^{-1}
\end{align*}
We have
\begin{align*}
&\partial_{\alpha}(G_{\alpha,n_j} \circ F_{\alpha,n_{j-1}}^{-1} \circ \dots \circ F_{\alpha,n_{0}}^{-1})\\
&=\partial_{\alpha}(G_{\alpha,n_j}) \circ F_{\alpha,n_{j-1}}^{-1} \circ \dots \circ F_{\alpha,n_{0}}^{-1}\\
&+(G_{\alpha,n_j} \circ F_{\alpha,n_{j-1}}^{-1} \circ \dots \circ F_{\alpha,n_{0}}^{-1})' \cdot \partial_{\alpha}(F_{\alpha,n_{j-1}}^{-1} \circ \dots \circ F_{\alpha,n_{0}}^{-1}) \\
&=\partial_{\alpha}(G_{\alpha,n_j}) \circ F_{\alpha,n_{j-1}}^{-1} \circ \dots \circ F_{\alpha,n_{0}}^{-1}\\
&+G_{\alpha,n_j}' \circ F_{\alpha,n_{j-1}}^{-1} \circ \dots \circ F_{\alpha,n_{0}}^{-1} \cdot (F_{\alpha,n_{j-1}}^{-1} \circ \dots \circ F_{\alpha,n_{0}}^{-1})' \cdot \partial_{\alpha}(F_{\alpha,n_{j-1}}^{-1} \circ \dots \circ F_{\alpha,n_{0}}^{-1})
\end{align*}
By (K2) from Lemma \ref{lem:k16} we have $\|G'_{\alpha,n_j}/G_{\alpha,n_j}\|_{\infty} \leq C$ and by (K5) we see that  $\|\partial_{\alpha}G_{\alpha,n_j}/G_{\alpha,n_j}\|_{\infty} \leq C$.  Using furthermore \eqref{eq:itinv} we obtain
\[
\|\partial_{\alpha}(G_{\alpha,n_j} \circ F_{\alpha,n_{j-1}}^{-1} \circ \dots \circ F_{\alpha,n_{0}}^{-1})\|_{\infty} \leq Cn_j^{-\left(1+\frac{1}{\alpha}\right)}\left((\logg n_j)^3+\sum_{m=0}^{j-1}(\logg n_m)^2\right)
\]
and hence
\begin{equation} \label{eq:thirdalpha}
\|\partial_{\alpha} (G_{\alpha,\underline{n}} \circ F_{\alpha,n_0}^{-1})\|_{\infty} \leq C\Pi_{j=1}^k n_j^{-\left(1+\frac{1}{\alpha}\right)} \sum_{j=1}^k \left((\logg n_j)^3+\sum_{m=0}^{j-1}(\logg n_m)^2\right)
\end{equation}
Using the previously computed bounds on the supremums and the bounds \eqref{eq:firstalpha}, \eqref{eq:secondalpha} and \eqref{eq:thirdalpha} on the partial derivative, we obtain for \eqref{eq:kfolda}
\begin{align*}
\sum_{n_k=1}^{\infty}\dots &\sum_{n_0=1}^{\infty} \partial_{\alpha}\{(h_{\alpha} \circ F_{\alpha, \underline{n}}^{-1} \circ F_{\alpha, n_0}^{-1}) (\hat \Psi_{\alpha} \circ F_{\alpha, \underline{n}}^{-1} \circ F_{\alpha, n_0}^{-1})  (\hat \Psi_{\alpha} \circ F_{\alpha, n_0}^{-1}) \nonumber \\
&\hspace{0.8cm}\times G_{\alpha, \underline{n}} \circ F_{\alpha, n_0}^{-1} \cdot G_{\alpha,n_0}\} \\
\leq C\sum_{n_k=1}^{\infty}\dots &\sum_{n_0=1}^{\infty}(n_0+1)^2(n_0\dots n_k)^{-\left(1+\frac{1}{\alpha}\right)}\Pi_{i=0}^k(\logg n_i)^2 \\
& +2(n_0+1)^2(n_0\dots n_k)^{-\left(1+\frac{1}{\alpha}\right)} \\
& +(n_0+1)^2(n_0\dots n_k)^{-\left(1+\frac{1}{\alpha}\right)}\Pi_{i=1}^k(\logg n_i)^3 \Pi_{i=0}^{k-1}(\logg n_i)^{2(k-i)} \\
& < \infty 
\end{align*}
for $\alpha \in (0,1/2)$ and thus $\alpha \mapsto \partial_{\alpha} \int P_{\alpha}\left(P_{\alpha}^k(\hat \Psi_{\alpha}h_{\alpha}) \cdot \hat \Psi_{\alpha} \right)$ is continuous.
\end{proof}

The next proposition claims that both the correlations and their partial derivatives with respect to $\alpha$ are summable. Then it follows from Propositions \ref{prop:cont} and \ref{prop:sum} that $\alpha \mapsto \tilde{\sigma}(\alpha) \in C^1([\alpha_-,\alpha_+])$ which concludes the proof of Theorem \ref{thm:main}.

\begin{proposition} \label{prop:sum}
The series
\[
\sum_{k \geq 1} \int \hat \Psi_{\alpha} \cdot \hat \Psi_{\alpha} \circ F_{\alpha}^k \cdot h_{\alpha} \: dm \quad \text{and} \quad \sum_{k \geq 1} \partial_{\alpha}\left(\int \hat \Psi_{\alpha} \cdot \hat \Psi_{\alpha} \circ F_{\alpha}^k \cdot h_{\alpha}\: dm\right)
\]
converge uniformly.	
\end{proposition}

\begin{proof}
We first show that 
\[
\sum_{k \geq 1} \int \hat \Psi_{\alpha} \cdot \hat \Psi_{\alpha} \circ F_{\alpha}^k \cdot h_{\alpha} \: dm=\sum_{k \geq 1} \int P_{\alpha}^k(\hat \Psi_{\alpha} h_{\alpha}) \cdot \hat \Psi_{\alpha} \: dm
\]	
converges uniformly. Our first observation is the following:
\begin{lemma} \label{lem:reg1}
\[
x \mapsto P_{\alpha}^{k}(\hat \Psi_{\alpha} h_{\alpha})(x) \in C^2
\]
\end{lemma}
\begin{proof}
Since $P_{\alpha}: C^2 \to C^2$, it is enough to prove the lemma for $k=1$.
\[
P_{\alpha}(\hat \Psi_{\alpha} h_{\alpha})=\sum_{n=1}^{\infty} (h_{\alpha} \circ F_{\alpha, n}^{-1}) (\hat \Psi_{\alpha} \circ F_{\alpha, n}^{-1})   G_{\alpha, n}
\]
where $h_{\alpha} \circ F_{\alpha, n}^{-1}$, $\hat \Psi_{\alpha} \circ F_{\alpha, n}^{-1}$ and $G_{\alpha, n}$ are continuously differentiable in $x$, and we have seen previously that $\|h_{\alpha}\|_{\infty} \leq K$, $\|\hat \Psi_{\alpha} \circ F_{\alpha, n}^{-1}\|_{\infty} \leq C(n+1)$ and $\| G_{\alpha, n}\|_{\infty} \leq C n^{-\left( 1+\frac{1}{\alpha}\right)}$. Thus 
\begin{align}
&\sum_{n=1}^{\infty} |(h_{\alpha} \circ F_{\alpha, n}^{-1}) (\hat \Psi_{\alpha} \circ F_{\alpha, n}^{-1})   G_{\alpha, n}| \leq C\sum_{n=1}^{\infty}(n+1) n^{-\left( 1+\frac{1}{\alpha}\right)} < \infty \label{eq:unifcont}
\end{align}
implying that the sum converges uniformly and $x \mapsto P_{\alpha}(\hat \Psi_{\alpha} h_{\alpha})(x)$ is continuous. 

We now study
\[
\sum_{n=1}^{\infty} \partial_x[(h_{\alpha} \circ F_{\alpha, n}^{-1}) (\hat \Psi_{\alpha} \circ F_{\alpha, n}^{-1})   G_{\alpha, n}]
\]
According to \cite[Theorem 2.1]{K16} we in fact have $\|h_{\alpha}\|_{C^2} \leq K$ and
\begin{align*}
(\hat \Psi_{\alpha} \circ F_{\alpha, n}^{-1})'(z)&=\left(\hat \psi_{\alpha}\left(\frac{z_{n}+1}{2} \right)+\sum_{j=0}^{n-1} \hat \psi_{\alpha}(f_{\alpha}^j(z_{n}))\right)'\\
&=\psi'\left(\frac{z_{n}+1}{2} \right)z_{n}'+\sum_{j=1}^{n} \psi'(z_{j})z_j'
\end{align*}
thus $\|(\hat \Psi_{\alpha} \circ F_{\alpha, n}^{-1})'\|_{\infty} \leq C\|\psi\|_{C^1}(n+1)$ (using that $z_j' \leq 1$ by (K0) of Lemma \ref{lem:k16}). 

Finally, $\|G_{\alpha_{n}}'/G_{\alpha,n}\|_{\infty} \leq C$ by (K2) from Lemma \ref{lem:k16}. Thus
\begin{align*}
&\sum_{n=1}^{\infty} |\partial_x[(h_{\alpha} \circ F_{\alpha, n}^{-1}) (\hat \Psi_{\alpha} \circ F_{\alpha, n}^{-1})   G_{\alpha, n}]| \leq  C\sum_{n=1}^{\infty} n^{-\left(1+\frac{1}{\alpha}\right)}(n+1) < \infty
\end{align*}
implying that $x \mapsto (P_{\alpha}(\hat \Psi_{\alpha} h_{\alpha}))'(x)$ is continuous. We do similar calculations to prove that $x \mapsto (P_{\alpha}(\hat \Psi_{\alpha} h_{\alpha}))''(x)$ is continuous. We compute
\[
\sum_{n=1}^{\infty} \partial^2_x[(h_{\alpha} \circ F_{\alpha, n}^{-1}) (\hat \Psi_{\alpha} \circ F_{\alpha, n}^{-1})   G_{\alpha, n}]
\]
and for fixed $n$ we get nine terms that are continuous individually. We can easily compute that $\|(h_{\alpha} \circ F_{\alpha, n}^{-1})''\|_{\infty} \leq C$, $\|(\hat \Psi_{\alpha} \circ F_{\alpha, n}^{-1})''\|_{\infty} \leq C\|\psi\|_{C^2}(n+1)$ (using that $|z_r''| \leq C$ by (K2) from Lemma \ref{lem:k16}), and $\|G_{\alpha,n}'' / G_{\alpha,n}\|_{\infty} \leq C$ by (K3) from Lemma \ref{lem:k16}. By using bounds computed previously, we obtain that
\begin{align*}
&\sum_{n=1}^{\infty} |\partial^2_x[(h_{\alpha} \circ F_{\alpha, n}^{-1}) (\hat \Psi_{\alpha} \circ F_{\alpha,n}^{-1})   G_{\alpha, n}]| \leq C\sum_{n=1}^{\infty}(n+1)n^{-\left(1+\frac{1}{\alpha}\right)} < \infty,
\end{align*}
implying that the sum converges uniformly.
\end{proof}

By \cite[Corollary 4.8]{K16}
\begin{equation} \label{eq:edc}
\|P_{\alpha}^{k}(\hat \Psi_{\alpha} h_{\alpha})\|_{C^i} \leq C(1-\theta)^k\|P_{\alpha}(\hat \Psi_{\alpha} h_{\alpha})\|_{C^i} \qquad i=1,2
\end{equation}
for some $\theta \in (0,1)$ and $C > 0$. We obtain $\|P_{\alpha}(\hat \Psi_{\alpha} h_{\alpha})\|_{C^1} \leq C$ by the computation of Lemma \ref{lem:reg1} and $\|\hat \Psi_{\alpha}\|_1 \leq C$ by a computation very similar to that in Proposition \ref{prop:l2}. 
Hence 
\begin{align*}
\sum_{k \geq 2} \left| \int P_{\alpha}^k(\hat \Psi_{\alpha} h_{\alpha}) \cdot \hat \Psi_{\alpha} \: dm \right | &\leq C\|P_{\alpha}(\hat \Psi_{\alpha}  h_{\alpha})\|_{C^1}\|\hat \Psi_{\alpha}\|_{L^1}\sum_{k \geq 2} (1-\theta)^k \\
&\leq C\sum_{k \geq 2} (1-\theta)^k,
\end{align*}
which proves the first statement of the proposition.

We now study
\begin{align*}
\sum_{k \geq 1} \partial_{\alpha}\left(\int \hat \Psi_{\alpha} \cdot \hat \Psi_{\alpha} \circ F_{\alpha}^k \cdot h_{\alpha}\: dm\right)= \sum_{k \geq 1} \int \partial_{\alpha}\left(P_{\alpha}^k(\hat \Psi_{\alpha} h_{\alpha}) \hat \Psi_{\alpha} \: dm \right).
\end{align*}
Write 
\begin{align*}
\int \partial_{\alpha}\left(P_{\alpha}^k(\hat \Psi_{\alpha} h_{\alpha}) \hat \Psi_{\alpha} \: dm \right) &= \int P_{\alpha}^k(\hat \Psi_{\alpha} h_{\alpha}) \partial_{\alpha}\hat \Psi_{\alpha} \: dm + \int \partial_{\alpha}[P_{\alpha}^k(\hat \Psi_{\alpha} h_{\alpha})] \hat \Psi_{\alpha} \: dm \\
&=(I) + (II)
\end{align*}
First,
\[
(I) = \int P_{\alpha}(P_{\alpha}^k(\hat \Psi_{\alpha} h_{\alpha}) \partial_{\alpha}\hat \Psi_{\alpha}
) \: dm =: \int P_{\alpha}(p_{\alpha}\partial_{\alpha}\hat \Psi_{\alpha}
) \: dm,
\]
where by Equation \eqref{eq:edc} we have $\|p_{\alpha}\|_{C^1} \leq C(1-\theta)^k$. Write 
\[
\int P_{\alpha}(p_{\alpha}\partial_{\alpha}\hat \Psi_{\alpha}
) \: dm = \int \sum_{n \geq 1} p_{\alpha}\circ F_{\alpha,n}^{-1 }\cdot \partial_{\alpha} \hat \Psi_{\alpha} \circ F_{\alpha,n}^{-1} \cdot G_{\alpha,n}
 \: dm
\]
Now 
\begin{align*}
\partial_{\alpha} \hat \Psi_{\alpha} \circ F_{\alpha,n}^{-1}=\partial_{\alpha} (\hat \Psi_{\alpha} \circ F_{\alpha,n}^{-1})-\hat \Psi_{\alpha}' \circ F_{\alpha,n}^{-1} \cdot  \partial_{\alpha} F_{\alpha,n}^{-1} \\
=\partial_{\alpha} (\hat \Psi_{\alpha} \circ F_{\alpha,n}^{-1})-(\hat \Psi_{\alpha} \circ F_{\alpha,n}^{-1})' (G_{\alpha,n})^{-1}\cdot  \partial_{\alpha} F_{\alpha,n}^{-1}
\end{align*}
thus
\[
\|(\partial_{\alpha} \hat \Psi_{\alpha} \circ F_{\alpha,n}^{-1})G_{\alpha,n}\|_{\infty} \leq C(n+1)n^{-\frac{1}{\alpha}}(\logg n)^2 
\]
by \eqref{eq:secondalpha} and (K6) from Lemma \ref{lem:k16}, thus
\[
 \sum_{n \geq 1} |p_{\alpha}\circ F_{\alpha,n}^{-1 }\cdot \partial_{\alpha} \hat \Psi_{\alpha} \circ F_{\alpha,n}^{-1} \cdot G_{\alpha,n}| \leq C(1-\theta)^k\sum_{n \geq 1} C(n+1)(\logg n)^2 n^{-\frac{1}{\alpha}} < \infty 
\]
for $\alpha < 1/2$, and
\begin{equation} \label{eq:firstcorr}
(I) \leq C_1(1-\theta)^k.
\end{equation}
Next,
\begin{align*}
(II)&= \int P_{\alpha}^k(\partial_{\alpha}[\hat \Psi_{\alpha} h_{\alpha}]) \hat \Psi_{\alpha} \: dm + \int \partial_{\alpha}[P_{\alpha}^k](\hat \Psi_{\alpha} h_{\alpha})\hat \Psi_{\alpha} \: dm \\
&=(IIa) + (IIb)
\end{align*}
With $(IIa)$ we can do the same argument as with $\int P_{\alpha}^k(\hat \Psi_{\alpha} h_{\alpha}) \hat \Psi_{\alpha} \: dm$, provided that $x \mapsto P_{\alpha}^k(\partial_{\alpha}\hat \Psi_{\alpha} h_{\alpha}+\hat \Psi_{\alpha} \partial_{\alpha} h_{\alpha})(x) \in C^1$. 

\begin{lemma}
\[
x \mapsto P_{\alpha}^k(\partial_{\alpha}\hat \Psi_{\alpha} h_{\alpha}+\hat \Psi_{\alpha} \partial_{\alpha} h_{\alpha})(x) \in C^1
\]
\end{lemma}

\begin{proof}
It is enough again to prove the statement for $k=1$. We write
\begin{align*}
&P_{\alpha}(\partial_{\alpha}[\hat \Psi_{\alpha} h_{\alpha}])\\
&=\sum_{n=1}^{\infty} (\partial_{\alpha}h_{\alpha} \circ F_{\alpha, n}^{-1}) (\hat \Psi_{\alpha} \circ F_{\alpha, n}^{-1})   G_{\alpha, n}+(h_{\alpha} \circ F_{\alpha, n}^{-1}) (\partial_{\alpha}\hat \Psi_{\alpha} \circ F_{\alpha, n}^{-1})   G_{\alpha, n}
\end{align*}
All terms are continuous in $x$, and the sum uniformly converges since
\begin{align*}
&\sum_{n=1}^{\infty} |(\partial_{\alpha}h_{\alpha} \circ F_{\alpha, n}^{-1}) (\hat \Psi_{\alpha} \circ F_{\alpha, n}^{-1})   G_{\alpha, n}+(h_{\alpha} \circ F_{\alpha, n}^{-1}) (\partial_{\alpha}\hat \Psi_{\alpha} \circ F_{\alpha, n}^{-1})   G_{\alpha, n}| \\
& \leq \sum_{n=1}^{\infty} (n+1)(\logg n)^2 n^{-\frac{1}{\alpha}} < \infty, 
\end{align*}

As for the continuity of $x \mapsto (P^k_{\alpha}(\partial_{\alpha}\hat \Psi_{\alpha} h_{\alpha}+\hat \Psi_{\alpha} \partial_{\alpha} h_{\alpha}))'(x)$, we write
\begin{align}
&(P_{\alpha}^k(\partial_{\alpha}[\hat \Psi_{\alpha} h_{\alpha}]))' \nonumber \\
&=\sum_{n=1}^{\infty}[((\partial_{\alpha}h_{\alpha})' \circ F_{\alpha, n}^{-1}) (\hat \Psi_{\alpha} \circ F_{\alpha, n}^{-1}) +(h_{\alpha}' \circ F_{\alpha, n}^{-1}) (\partial_{\alpha}\hat \Psi_{\alpha} \circ F_{\alpha, n}^{-1})]G_{\alpha, n}^2 \nonumber \\
&+[(\partial_{\alpha}h_{\alpha} \circ F_{\alpha, n}^{-1}) (\hat \Psi_{\alpha} \circ F_{\alpha, n}^{-1})' +(h_{\alpha} \circ F_{\alpha, n}^{-1}) (\partial_{\alpha}\hat \Psi_{\alpha} \circ F_{\alpha, n}^{-1})'] G_{\alpha, n} \nonumber \\
&+ [((\partial_{\alpha}h_{\alpha}) \circ F_{\alpha, n}^{-1}) (\hat \Psi_{\alpha} \circ F_{\alpha, n}^{-1}) +(h_{\alpha}\circ F_{\alpha, n}^{-1}) (\partial_{\alpha}\hat \Psi_{\alpha} \circ F_{\alpha, n}^{-1})]G'_{\alpha, n} \label{eq:dersum}
\end{align}
We compute that
\begin{align*}
(\partial_{\alpha}[\hat \Psi_{\alpha}  \circ F_{\alpha, n}^{-1}])'&=\psi' \left(\frac{z_{n}+1}{2} \right)z_{n}'\partial_{\alpha}z_{n} + \psi \left(\frac{z_{n}+1}{2} \right)\partial_{\alpha}z_{n}' \\
&+\sum_{i=1}^{n-1} \psi' \left(z_{i} \right)z_{i}'\partial_{\alpha}z_i + \psi \left(z_i \right)\partial_{\alpha}z_i',
\end{align*}
and since $|z_i'|, |\partial_{\alpha}z_i'| \leq C$ and $|\partial_{\alpha}z_i'| \leq C (\logg i)^3$ we have $\|(\partial_{\alpha}[\hat \Psi_{\alpha}  \circ F_{\alpha, n}^{-1}])'\|_{\infty} \leq C (n+1)(\logg n)^3$ and thus
\begin{align*}
&\|(\partial_{\alpha}\hat \Psi_{\alpha} \circ F_{\alpha, n}^{-1})'G_{\alpha,n}\|_{\infty} \leq \|(\partial_{\alpha}[\hat \Psi_{\alpha}  \circ F_{\alpha, n}^{-1}])'\|_{\infty}\|G_{\alpha,n}\|_{\infty} \\
&+ \|(\hat \Psi_{\alpha}'  \circ F_{\alpha, n}^{-1})'G_{\alpha,n}\|_{\infty}\|\partial_{\alpha}F_{\alpha, n}^{-1}\|_{\infty}+ \|(\hat \Psi_{\alpha}'  \circ F_{\alpha, n}^{-1}) G_{\alpha,n}\|_{\infty}\|\partial_{\alpha}(F_{\alpha, n}^{-1})'\|_{\infty} \\
& \leq C(n+1) \left[(\logg n)^3n^{-\left(1+\frac{1}{\alpha}\right)}+(\logg n)^2n^{-\frac{1}{\alpha}}+ (\logg n)^5n^{-\left(1+\frac{1}{\alpha}\right)} \right]
\end{align*}
thus
\[
\eqref{eq:dersum} \leq C\sum_{n=1}^{\infty}n^{-\left(1+\frac{1}{\alpha} \right)} (n+1)(\logg n)^5n^{-\frac{1}{\alpha}}< \infty.
\]
\end{proof}
Returning to $(IIa)$, we write
\[
(IIa) = \int P_{\alpha}(P_{\alpha}^k(\partial_{\alpha}[\hat \Psi_{\alpha} h_{\alpha}]) \hat \Psi_{\alpha}
) \: dm =: \int P_{\alpha}(q_{\alpha}\hat \Psi_{\alpha}
) \: dm,
\]
and we have 
\begin{align*}
\left| \int P_{\alpha}(q_{\alpha}\hat \Psi_{\alpha})dm\right| &\leq  \sum_{n \geq 1} |q_{\alpha}\circ F_{\alpha,n}^{-1 }\cdot  \hat \Psi_{\alpha} \circ F_{\alpha,n}^{-1} \cdot G_{\alpha,n}| \\
&\leq C(1-\theta)^k\sum_{n \geq 1} C(n+1)n^{-\left(1+\frac{1}{\alpha}\right)} 
\end{align*}
giving
\begin{equation} \label{eq:secondcorr}
(IIa) \leq C_{2}(1-\theta)^k.
\end{equation}
Finally, we study $(IIb)$. Write
\[
\partial_{\alpha}[P_{\alpha}^k] = \lim_{\alpha \to \alpha'}\frac{1}{\alpha-\alpha'}\sum_{i=1}^k P_{\alpha}^{k-i}(P_{\alpha}-P_{\alpha'})P_{\alpha'}^{i-1}=:\sum_{i=1}^k P_{\alpha}^{k-i}Q_{\alpha}P_{\alpha}^{i-1},
\]
using similar notation to \cite{K16}.
Hence 
\begin{align*}
\int \partial_{\alpha}[P_{\alpha}^k](\hat \Psi_{\alpha} h_{\alpha}) \hat \Psi_{\alpha} \: dm=\int \sum_{i=1}^k P_{\alpha}^{k-i}Q_{\alpha}P_{\alpha}^{i-1}(\hat \Psi_{\alpha} h_{\alpha}) \hat \Psi_{\alpha} \: dm
\end{align*}
We first look at the $i \geq 2$ terms of the sum. Then
\begin{align*}
\sum_{i=2}^k |P_{\alpha}^{k-i}Q_{\alpha}P_{\alpha}^{i-1}(\hat \Psi_{\alpha} h_{\alpha}) \hat \Psi_{\alpha}| &\leq Ck(1-\theta)^{k-i}\|Q_{\alpha}P_{\alpha}^{i-1}(\hat \Psi_{\alpha} h_{\alpha})\|_{C^1}|\hat \Psi_{\alpha}| \\
&\leq  Ck(1-\theta)^{k-i}\|P_{\alpha}^{i-1}(\hat \Psi_{\alpha} h_{\alpha})\|_{C^2}|\hat \Psi_{\alpha}| \\
&\leq  Ck(1-\theta)^{k-2}\|P_{\alpha}(\hat \Psi_{\alpha} h_{\alpha})\|_{C^2}|\hat \Psi_{\alpha}|
\end{align*}
using \cite[Lemma 4.1]{K16} in the second step, and \eqref{eq:edc} in the first and third step. We obtain $\|P_{\alpha}(\hat \Psi_{\alpha} h_{\alpha})\|_{C^2} \leq C$ by the computation of Lemma \ref{lem:reg1} and $\|\hat \Psi_{\alpha}\|_1 \leq C$ by a computation very similar to that in Proposition \ref{prop:l2}. Thus
\[
\left| \int \sum_{i=2}^k P_{\alpha}^{k-i}Q_{\alpha}P_{\alpha}^{i-1}(\hat \Psi_{\alpha} h_{\alpha}) \hat \Psi_{\alpha} \: dm \right| \leq Ck(1-\theta)^{k-2}.
\]
As for the $i=1$ term, we show that $Q_{\alpha}(\hat \Psi_{\alpha}h_{\alpha}) \in C^1$. Indeed,
\begin{align*}
Q_{\alpha}(\hat \Psi_{\alpha}h_{\alpha}) = \sum_{n \geq 1} &\partial_{\alpha}[\hat \Psi_{\alpha} \circ F_{\alpha,n}^{-1} \cdot h_{\alpha} \circ F_{\alpha,n}^{-1} \cdot G_{\alpha,n}] \\
=\sum_{n \geq 1} &\partial_{\alpha}[\hat \Psi_{\alpha} \circ F_{\alpha,n}^{-1}] \cdot h_{\alpha} \circ F_{\alpha,n}^{-1} \cdot G_{\alpha,n} \\
&+\hat \Psi_{\alpha} \circ F_{\alpha,n}^{-1} \cdot \partial_{\alpha}[h_{\alpha} \circ F_{\alpha,n}^{-1}] \cdot G_{\alpha,n} \\
&+ \hat \Psi_{\alpha} \circ F_{\alpha,n}^{-1} \cdot h_{\alpha} \circ F_{\alpha,n}^{-1} \cdot \partial_{\alpha}G_{\alpha,n}
\end{align*}
Each term is continuous in $x$, and by previously computed bounds we can upper bound the sum by
\[
C\sum_{n \geq 1} (n+1)(1+(\logg n)^2+(\logg n)^3)n^{-\left(1+\frac{1}{\alpha}\right)} < \infty,
\]
hence it converges uniformly. By similar argument we also get that $(Q_{\alpha}(\hat \Psi_{\alpha}h_{\alpha}))'$ is continuous: we compute 
\[
Q_{\alpha}(\hat \Psi_{\alpha}h_{\alpha}) = \sum_{n \geq 1} \partial_x\partial_{\alpha}[\hat \Psi_{\alpha} \circ F_{\alpha,n}^{-1} \cdot h_{\alpha} \circ F_{\alpha,n}^{-1} \cdot G_{\alpha,n}],
\]
and we see that for fixed $n$ each of the nine terms are continuous in $x$. By using previously computed bounds, we again upper bound the sum by  
\[
C\sum_{n \geq 1} (n+1)(\logg n)^3n^{-\left(1+\frac{1}{\alpha}\right)} < \infty,
\]
proving that it converges uniformly. 

Then by \eqref{eq:edc} we have
\[
\|P_{\alpha}^{k-1}Q_{\alpha}(\hat \Psi_{\alpha}h_{\alpha})\|_{C^1} \leq C(1-\theta)^{k-1},
\]
thus
\begin{equation} \label{eq:thirdcorr}
(IIb) \leq  \begin{cases}
C_3 &\text{ if } k=1 \\
C_3k(1-\theta)^{k-2} &\text{ if } k \geq 2
\end{cases}
\end{equation}
We can conclude the proof of the proposition by combining \eqref{eq:firstcorr}, \eqref{eq:secondcorr} and  \eqref{eq:thirdcorr}
to obtain
\[
\sum_{k \geq 1} \int \partial_{\alpha}\left(P_{\alpha}^k(\hat \Psi_{\alpha} h_{\alpha}) \hat \Psi_{\alpha} \: dm \right) \leq C+C\sum_{k \geq 2} k(1-\theta)^{k-2} < \infty. 
\]
\end{proof}

\section{Concluding remarks}

A number of further questions would be interesting to study in the future. The most straightforward one would be considering an observable for which $\psi(0)=0$ and $\alpha \in (1/2,1)$. In this case the central limit theorem holds \cite{G04} and according to a 2002 announcement by Hu, correlations are summable -- making the Green--Kubo formula well-defined. However, a proof of this statement is not possible to track down in the literature. Provided that this in fact holds, we can expect that $\alpha \mapsto \sigma^2(\alpha) \in C^1[\alpha_-,\alpha_+]$ for any $1/2 < \alpha_- < \alpha_+< 1$. However, the calculations in these notes made use of $\alpha < 1/2$ in several places so the generalization is not completely straightforward. 

For a general observable $\psi$, we have $(S_n \psi)/n^{\alpha}$ converging in distribution to a random variable with a stable law of index $1/\alpha$ \cite{G04}. In this case the first task would be to give a proper definition for the diffusion coefficient. Venturing to the regime of $\alpha \geq 1$, $f_{\alpha}$ preserves a $\sigma$-finite measure, and the corresponding anomalous diffusion calls for the definition of a generalized diffusion coefficient, possibly along the lines of \cite{KK07}. It would be an intriguing task to check rigorously the discontinuities and fractal properties of the diffusion coefficient uncovered by numerics in \cite{KK07}.

Returning to the setting of the current paper, another interesting question would be to study further regularity of the drift- and diffusion coefficient as a function of $\alpha$. $C^2$-smoothness of the drift coefficient is sometimes called \emph{quadratic response} in the literature and has essentially only been studied in the uniformly expanding setting \cite{GS20}. The first task would be to clear quadratic response for the LSV map, then one could move on to study higher order regularity of the diffusion coefficient.

\begin{figure}
	\centering
	\begin{tikzpicture}
\draw[scale=2] (0,0) -- (2,0);
	\draw[scale=2] (0,0) -- (0,2);
	\draw[scale=2] (2,0) -- (2,2) -- (0,2);
	%\draw[scale=2,dotted] (1,0) -- (1,2);
	\draw[scale=2,dotted] (0,1.4) -- (2,1.4);
	\foreach \x/\xtext in {0/0,4/1}
	\draw[shift={(\x,0)},scale=2] (0pt,2pt) -- (0pt,-2pt) node[below] {$\xtext$};
	\foreach \y/\ytext in {2.8/H}
	\draw[shift={(0,\y)},scale=2] (2pt,0pt) -- (-2pt,0pt) node[left] {$\ytext$};
	\draw[domain=0:1, smooth, variable=\x,thick,scale=2] plot ({\x}, {2*(\x/2+2^0.9*(\x/2)^1.9)});
	\draw[thick,scale=2] (1,0) -- (2,1.4);
	\end{tikzpicture}
	\caption{LSV-type map with a non-full branch} \label{Fig:LSV2}
\end{figure}
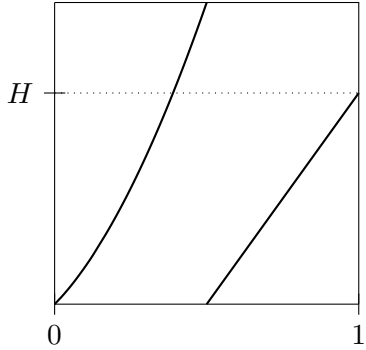

Another possible direction of generalization would be to consider LSV-type maps with a non-full branch, such as the one displayed on Figure \ref{Fig:LSV2}. In this setting one could study the regularity of the induced map's diffusion coefficient by the methods developed in \cite{KKH08} to obtain a log-Lipschitz modulus of continuity, namely
\[
|\tilde{\sigma}^2(\alpha)-\tilde{\sigma}^2(\alpha')| \leq C|\alpha-\alpha'|(1+|\log|\alpha-\alpha'||)^2, \quad \alpha,\alpha' \in [\alpha_-,\alpha_+]
\]
which would be inherited by $\sigma^2(\alpha)$, provided that $\alpha \mapsto \int \tau_{\alpha}d\mu_{\alpha}$ has the same type of (or better) regularity. For this, the regularity of  $\alpha \mapsto \tau_{\alpha}$ and $\alpha \mapsto h_{\alpha}$ has to be studied. According to \cite{K82}, $|h_{\alpha}-h_{\alpha'}|_{L^1} \leq C|\alpha-\alpha'||\log|\alpha-\alpha'||$ (the infinite number of branches does not cause a difficulty in the proof), but the regularity of  $\alpha \mapsto \tau_{\alpha}$ remains to be studied.

It should be possible to study the diffusion coefficient as a function of both the parameter $\alpha$ and the height $H$ of the second branch, and obtain 
\[
|\tilde{\sigma}^2(\alpha,H)-\tilde{\sigma}^2(\alpha',H')| \leq C(|\alpha-\alpha'|+|H-H'|)(1+|\log(|\alpha-\alpha'|+|H-H'|)|)^2
\]
for all $(\alpha,H), (\alpha', H') \in [\alpha_-,\alpha_+] \times [H_-,H_+]$ where $0 < \alpha_- < \alpha_+ < 1/2$ and  $1/2 < H_- < H_+ < 1$ (so that the second branch is expanding.)

\section*{Statements and declarations}
\subsection*{Funding and competing interests} The author has no relevant financial or non-financial interests to disclose.
\subsection*{Data availability statement} Not applicable.

\end{document}